\pdfoutput=1
\RequirePackage{ifpdf}
\ifpdf % We are running pdfTeX in pdf mode
\documentclass[pdftex]{sigma}
\else
\documentclass{sigma}
\fi

\numberwithin{equation}{section}

\newtheorem{Theorem}{Theorem}[section]
\newtheorem{Proposition}[Theorem]{Proposition}
 { \theoremstyle{definition}
\newtheorem{Definition}[Theorem]{Definition}
\newtheorem{Remark}[Theorem]{Remark} }

\begin{document}
\allowdisplaybreaks

\newcommand{\arXivNumber}{1811.01517}

\renewcommand{\PaperNumber}{022}

\FirstPageHeading

\ShortArticleName{On a Yang--Mills Type Functional}

\ArticleName{On a Yang--Mills Type Functional}

\Author{C\u{a}t\u{a}lin GHERGHE}

\AuthorNameForHeading{C.~Gherghe}

\Address{University of Bucharest, Faculty of Mathematics and Computer Science,\\ Academiei 14, Bucharest, Romania}
\Email{\href{mailto:catalinliviu.gherghe@gmail.com}{catalinliviu.gherghe@gmail.com}}

\ArticleDates{Received November 13, 2018, in final form February 27, 2019; Published online March 21, 2019}

\Abstract{We study a functional that derives from the classical Yang--Mills functional and Born--Infeld theory. We establish its first variation formula and prove the existence of critical points. We also obtain the second variation formula.}


\Classification{58E15; 81T13; 53C07}

\section{Motivations}

Let $u\colon \Omega\subset \mathbb{R}^n\rightarrow \mathbb{R}$ be a smooth function. Then the graph of $u$
 \begin{gather*}
G_u=\big\{(x,z)\in \mathbb{R}^{n+1}\,\vert\, z=u(x), \, x\in\Omega\big\},
\end{gather*}
is a minimal hypersurface if and only if satisfies the following differential equation
\begin{gather}\label{1}
\operatorname{div}\left(\frac{\nabla u}{\sqrt{1+\vert\nabla u\vert^2}}\right)=0.
\end{gather}

In 1970 Calabi, in a paper in which he studied examples of Bernstein problems, noticed that if $n=2$, $u$ is an $F$-harmonic map for the function $F(t)=\sqrt{1+2t}-1$. We recall that $u$ is an $F$-harmonic map if it is a critical point of the following functional
\begin{gather*}
E_F(u)=\int_{\mathbb{R}^2}F\left(\frac{\Vert du\Vert^2}{2}\right)\vartheta_g,
\end{gather*}
with respect to any compactly supported variation, $\Vert du\Vert^2$ being the Hilbert--Schmidt norm.

Following Calabi's ideas, Yang and then Sibner showed that for $n=3$, the equation (\ref{1}) is equivalent, over a simply connected domain, to the vector equation
\begin{gather*}
\nabla \times \left(\frac{\nabla\times A}{\sqrt{1+\vert\nabla\times A\vert^2}}\right)=0,
\end{gather*}
which arises in the nonlinear electromagnetic theory of Born and Infeld. Here $A$ is a vector field in $\mathbb{R}^3$ and $\nabla\times(\;\cdot\;)$ is the curl of $(\;\cdot\;)$. Born--Infeld theory is of contemporary interest due to its relevance in string theory.

This observation lead Yang to give a generalized treatment of equation (\ref{1}), expressed in terms of differential forms, as follows:
\begin{gather}\label{2}
\delta\left(\frac{d\omega}{\sqrt{1+\Vert d\omega\Vert^2}}\right)=0,
\end{gather}
for any $\omega\in A^p\big(\mathbb{R}^4\big)$. It is not very difficult to verify that the solution of equation~(\ref{2}) is a~critical point of the following integral
\begin{gather*}
\int_{\mathbb{R}^4}\left(\sqrt{1+\Vert d\omega \Vert^2}-1\right)\vartheta_g.
\end{gather*}

These facts give us the motivation to study a similar functional, namely the Yang--Mills--Born--Infeld functional
\begin{gather*}
{\rm YM}_{\rm BI}(D) = \int_M\Big(\sqrt{1+\big\Vert R^D\big\Vert^2}-1 \Big)\vartheta_g,
\end{gather*}
 defined more generally on Riemannian manifolds. The definition of the above functional is similar to the definition of the well-known Yang--Mills functional (see also \cite{Do}).

The paper is organized as follows. In Section~\ref{section1} we give some preliminaries and define the functional. In Section~\ref{section2} we derive its Euler--Lagrange equations and we obtain the main result of the paper (Theorem~\ref{theorem2}). In dimension $\geq 5$, we give criteria for which a metric is conformal to a metric with respect to which a $G$-connection is critical for ${\rm YM}_{\rm BI}$. Section~\ref{section3} is devoted to a~conservation law of the functional. Finally in Section~\ref{section4} we derive the second variation formula.

\section{The functional}\label{section1}
 Let $E$ be a smooth real vector bundle over a compact $n$-dimensional Riemannian manifold $(M^n,g)$, such that its structure group $G$ is a compact Lie subgroup of the orthogonal group~$O(n)$.

 For any vector bundle $F$ over $M$ we denote by $\Gamma(F)$ the space of smooth cross sections of $F$ and for each $p\ge 0$ we denote by $\Omega^p(F) = \Gamma(\Lambda^pT^*M\otimes F)$ the space of all smooth $p$-forms on~$M$ with values in~$F$. Note that $\Omega^0(F)=\Gamma(F)$.

 A connection $D$ on the vector bundle $E$ is defined by specifying a covariant derivative, that is a linear map
 \begin{gather*}
D\colon \ \Omega^0(E)\rightarrow \Omega^1(E),
\end{gather*}
such that $D(fs)=df\otimes s + fDs$, for any section $s\in \Omega^0(E)$ and any smooth function $f\in C^{\infty}(M)$.

A connection $D$ is called a $G$-connection if the natural extension of $D$ to tensor bundles of $E$ annihilates the tensors which define the $G$-structure. We denote by ${\cal C}(E)$ the space of all smooth $G$-connections~$D$ on~$E$.

 Given a connection on $E$, the map $D\colon \Omega^0(E)\rightarrow \Omega^1(E)$ can be extended to a generalized de Rham sequence
 \begin{gather*}
\Omega^0(E)\stackrel{d^D=D}{\longrightarrow} \Omega^1(E)\stackrel{d^D}{\longrightarrow} \Omega^2(E)\stackrel{d^D}{\longrightarrow}\cdots.
 \end{gather*}
For each $G$-connection $D$ of the vector bundle $E$, the curvature tensor of $D$, denoted by $R^D$, is determined by $\big(d^D\big)^2 \colon\Omega ^0(E)\rightarrow \Omega^2(E)$. If we suppose that $E$ carries an inner product compatible with~$G$, it is easy to see that $R^D\in \Omega^2(g_E)$, where $g_E\subset \operatorname{End}(E)$ is the subbundle of skew-symmetric endomorphisms of~$E$.

Given metrics on $M$ and $E$, there are naturally induced metrics on all associated bundles, such as $\Lambda^pT^*M\otimes \operatorname{End}(E)$:
\begin{gather*}
\langle\varphi , \psi \rangle_x= \sum_{1<i_1<\dots <i_p<n}\big\langle\varphi^t (e_{i_1}, \ldots,e_{i_p}), \psi( e_{i_1}, \ldots ,e_{i_p})\big\rangle,
\end{gather*}
 where, for any point $x\in M$, $\{e_i\}^n_{i=1}$ is an orthonormal basis of $T_xM$ with respect to the metric~$g$. The pointwise inner product gives an $L^2$-norm on $\Omega^p(E)$ by setting
\begin{gather*}
(\varphi, \psi)=\int_M\langle\varphi , \psi \rangle \vartheta_g.
\end{gather*}
With respect to this norm, the formal adjoint of $d^D$ it is denoted by $\delta^D$ (the coderivative) and satisfies
\begin{gather*}
\big(d^D\varphi, \psi\big)=\big(\varphi,\delta^D\psi\big).
\end{gather*}

In particular, for any $G$-connection $D$, the norm of the curvature $R^D$ is defined by
\begin{gather*}
\big\Vert R^D \big\Vert^2_x=\sum_{i<j}\big\Vert R^D_{e_i,e_j}\big\Vert^2_x,
\end{gather*}
for any point $x\in M$ and any orthonormal basis $\{e_i\}_{i=\overline{1,n}}$ on $T_xM$. The norm of $R^D_{e_i,e_j}$ is the usual one on $\operatorname{End}(E)$, namely
$\langle A, B \rangle = \frac12\operatorname{tr}\big(A^t\circ B\big)$.

 We are able to define the Yang--Mills--Born--Infeld functional ${\rm YM}_{\rm BI} \colon {\cal C}(E) \rightarrow \mathbb{R}$ (see also~\cite{Do}) by
 \begin{gather*}
{\rm YM}_{\rm BI}(D) = \int_M\left(\sqrt{1+\big\Vert R^D\big\Vert^2}-1 \right)\vartheta_g.
\end{gather*}

\section{The first variation formula. Existence result}\label{section2}

In the following we shall derive the Euler--Lagrange equations of the functional ${\rm YM}_{\rm BI}$. These equations were also obtained in \cite{Do} for the $F$-Yang--Mills functional.

 \begin{Theorem}\label{theorem1}The first variation formula of the functional ${\rm YM}_{\rm BI}$ is given by
 \begin{gather*}
\frac {d}{dt}\Bigg\vert_{t=0}{\rm YM}_{\rm BI}\big(D^t\big) = \int_M \left\langle B,\delta^D\left(\frac{1}{\sqrt{1+\big\Vert R^D\big\Vert^2}}R^D\right)\right\rangle\vartheta_g,
\end{gather*}
 where
 \begin{gather*}
B = \frac d{dt}\bigg\vert_{t=0}D^t.
\end{gather*}
Consequently, $D$ is a critical point of ${\rm YM}_{\rm BI}$ if and only if
 \begin{gather*}
\delta^D\left(\frac{1}{\sqrt{1+\big\Vert R^D\big\Vert^2}}R^D\right)=0,
\end{gather*}
which are the Euler--Lagrange equations of ${\rm YM}_{\rm BI}$.
 \end{Theorem}

\begin{proof} Let $D$ a $G$-connection $D\in{\cal C}(E)$ and consider a smooth curve $D^t=D+\alpha^t$ on ${\cal C}(E)$, $t\in (-\epsilon , \epsilon)$, such that $\alpha^0=0$, where $\alpha^t\in \Omega^1 (g_E)$. The corresponding curvature is given by
 \begin{gather*}
R^{D^t}=R^D + d^D\alpha^t + \tfrac12 \big[\alpha^t \wedge \alpha^t\big],
\end{gather*}
 where we define the bracket of $g_E$-valued 1 forms $\varphi$ and $ \psi $ by the formula $[\varphi \wedge \psi ](X, Y) = [\varphi(X), \psi (Y)] - [\varphi (Y), \psi (X)]$ for any vector fields $X,Y \in\Gamma(TM)$. Indeed for any vector fields $X,Y \in\Gamma(TM)$ and $u\in \Gamma(E)$ we have
 \begin{gather*}
 R^{D^t}(X,Y)(u) = D_X^t\big(D_Y^tu\big)-D_Y^t\big(D_X^tu\big)-D^t_{[X,Y]}u\\
 \hphantom{R^{D^t}(X,Y)(u)}{} = D_X^t\big(D_Yu+\alpha^t(Y)(u)\big)-D_Y^t\big(D_Xu+\alpha^t(X)(u)\big) \\
 \hphantom{R^{D^t}(X,Y)(u)=}{} -D_X^t\big(D_{[X,Y]}u+\alpha^t([X,Y])(u)\big)\\
 \hphantom{R^{D^t}(X,Y)(u)}{} = D_X\big(D_Yu+\alpha^t(Y)(u)\big)+\alpha^t(X)\big(D_Yu+\alpha^t(Y)(u)\big)\\
\hphantom{R^{D^t}(X,Y)(u)=}{} - D_Y\big(D_Xu+\alpha^t(X)(u)\big)-\alpha^t(Y)\big(D_Xu+\alpha^t(X)(u)\big)\\
\hphantom{R^{D^t}(X,Y)(u)=}{} -D_{[X,Y]}u-\alpha([X,Y])(u)\\
\hphantom{R^{D^t}(X,Y)(u)}{} = R^D(X,Y)(u)+D_X\big(\alpha^t(Y)(u)\big)-\alpha^t(Y)(D_Xu)\\
\hphantom{R^{D^t}(X,Y)(u)=}{} -\big( D_Y(\alpha^t(X)(u))-\alpha^t(X)(D_Yu)\big)-\alpha^t([X,Y])(u)\\
\hphantom{R^{D^t}(X,Y)(u)=}{} +\alpha^t(X)\big(\alpha^t(Y)(u)\big)-\alpha^t(Y)\big(\alpha^t(X)(u)\big)\\
\hphantom{R^{D^t}(X,Y)(u)}{} = R^D(X,Y)(u)+\big(D_X\big(\alpha^t(Y)\big)(u)\big)-\big(D_Y\big(\alpha^t(X)\big)(u)\big)\\
\hphantom{R^{D^t}(X,Y)(u)=}{} -\alpha^t([X,Y])(u)+\tfrac12\big[\alpha^t\wedge\alpha^t\big](X,Y)(u)\\
\hphantom{R^{D^t}(X,Y)(u)}{} = R^D(X,Y)(u)+\big(d^D\alpha^t\big)(X,Y)(u)+\tfrac12\big[\alpha^t\wedge\alpha^t\big](X,Y)(u).
 \end{gather*}

 Then we have
\begin{gather*}
\frac {d}{dt}\bigg\vert_{t=0} \Big(\sqrt{1+\big\Vert R^{D^t}\big\Vert^2}-1 \Big) =\frac{1}{\sqrt{1+\big\Vert R^D\big\Vert^2}}\frac {d}{dt}\bigg|_{t=0}\frac12\big\|R^{D^t}\big\|^2\\
\hphantom{\frac {d}{dt}\bigg\vert_{t=0} \Big(\sqrt{1+\big\Vert R^{D^t}\big\Vert^2}-1 \Big)}{}
 =\frac{1}{\sqrt{1+\big\Vert R^D\big\Vert^2}}\left\langle \frac{d}{dt}R^{D^t},R^D\right\rangle\bigg\vert_{t=0}\\
\hphantom{\frac {d}{dt}\bigg\vert_{t=0} \Big(\sqrt{1+\big\Vert R^{D^t}\big\Vert^2}-1 \Big)}{}
 =\frac{1}{\sqrt{1+\big\Vert R^D\big\Vert^2}}\big\langle d^DB,R^D\big\rangle,
\end{gather*}
where $B = \frac d{dt}\big\vert_{t=0}D^t \in \Omega^1(g_E)$.

 Thus we obtain
\begin{gather*}
\frac {d}{dt}\bigg\vert_{t=0} {\rm YM}_{\rm BI}\big(D^t\big) = \int_M\frac{1}{\sqrt{1+\big\Vert R^D\big\Vert^2}}\big\langle d^DB,R^D\big\rangle \vartheta_g \\
\hphantom{\frac {d}{dt}\bigg\vert_{t=0} {\rm YM}_{\rm BI}\big(D^t\big)}{} =\int_M \left\langle B, \delta^D\left(\frac{1}{\sqrt{1+\big\Vert R^D\big\Vert^2}}R^D\right)\right\rangle\vartheta_g.\tag*{\qed}
\end{gather*}\renewcommand{\qed}{}
\end{proof}

After deriving the Euler--Lagrange equations, we look for their solutions. We next prove an existence result for the critical points of the functional ${\rm YM}_{\rm BI}$.

\begin{Theorem}\label{theorem2}Let $(M,g)$ be an $n$-dimensional compact Riemannian manifold with $n\ge 5$, $G$~a~compact Lie group, and $E$ a smooth $G$-vector bundle over $M$. Then there exists a Riemannian metric~$\tilde g$ on $M$ in the conformal class of~$g$, and a $G$-connection $D$ on $E$ such that $D$ is a~critical point of the functional ${\rm YM}_{\rm BI}$.
 \end{Theorem}

\begin{proof} We prove the theorem in two steps.

{\bf Step 1.} Consider the functional $F_p\colon {\cal C}(E) \rightarrow \mathbb{R}$, defined by
\begin{gather*}
F_p(D)=\frac12\int_M\big(1+\big\Vert R^D\big\Vert^2_g\big)^{(p-2)/2}\vartheta_g.
\end{gather*}
By \cite{Uh} this functional satisfies the Palais--Smale conditions and attains the minimum if $2p>n$. The Euler--Lagrange equation associated to $F_p(D)$ is
\begin{gather*}
\delta^D_g\big( \big(1+\big\Vert R^D\big\Vert^2_g\big)^{(p-2)/2}R^D\big)=0.
\end{gather*}
This equation has a solution $D$ for $2p>n$. Define the function $f\colon M\rightarrow \mathbb{R}$ by $f=\big(1+\big\Vert R^D\big\Vert^2_g\big)^{(p-2)/n-4}$ and the metric $\overline{g}=fg$, conformally equivalent to $g$. As $\delta^D_g\left( f^{(n-4)/2}R^D\right)=0$, it is easy to see that $\delta^D_{\overline g}\big( R^D\big)=0$. Hence there exists a Riemannian metric $\overline g$ on $M$, conformaly equivalent to $g$, and a $G$-connection $D$ on $E$ such that $D$ is a Yang--Mills connection with respect to $\overline{g}$.

{\bf Step 2.} Now we look for a ``good'' function $\sigma$ such that $\tilde{g}=\sigma^{-1}g$. Taking into account the first step, we can start with an Yang--Mills connection~$D$ with respect to the metric~$g$. It is clear that
\begin{gather*}\delta^D_gR^D=0 \qquad \mbox{if and only if} \quad \delta_{\tilde g}^D\big(\sigma^{\frac{n-4}{2}}R^D\big)=0,
\end{gather*}
for any $G$-connection.

The function $\sigma$ is good if it satisfies the following functional equation
\begin{gather*}
\sigma^{\frac{n-4}{2}}=\frac{1}{\sqrt{1+\sigma^2\big\Vert R^D\big\Vert_g^2}} \left(=\frac{1}{\sqrt{1+\big\Vert R^D\big\Vert_{\tilde{g}}^2}}\right).
\end{gather*}
So, what we have to do next is to solve the above functional equation.

Let $h\colon [0,\infty)\rightarrow [0,\infty)$ given by $h(t)=\sqrt{1+2t}-1$. It is clear that its derivative is a strictly decreasing function and let $H\colon (0,1]\rightarrow [0,\infty)$ be its smooth inverse. Consider the smooth function $F \colon (0,1]\rightarrow [0,\infty)$ given by
\begin{gather*}
F(y) = \frac {H\big(y^{(n-4)/2}\big)}{y^2}.
\end{gather*}
It is not difficult to prove that $F$ is invertible. Denote by $\Phi\colon [0,\infty)\rightarrow (0,1]$ the smooth inverse of $F$. We define the positive smooth function $\sigma$ by
\begin{gather*}
\sigma= \Phi \big(\tfrac12\big\|R^D\big\|^2_g\big).
\end{gather*}
We then have
\begin{gather*}
0 = \delta^D_{\tilde g}\big(\sigma^{(n-4)/2}R^D\big)= \delta^D_{\tilde g}\Big(\big(\Phi \big(\tfrac12\big\|R^D\big\|^2_g\big)\big)^{(n-4)/2}R^D\Big)\\
\hphantom{0}{} =\delta^D_{\tilde g}\left(\frac{1}{\sqrt{1+\sigma^2\big\Vert R^D\big\Vert_g^2}}R^D\right)=\delta^D_{\tilde g}\left(\frac{1}{\sqrt{1+\big\Vert R^D\big\Vert_{\tilde{g}}^2}}R^D\right),
\end{gather*}
which proves that the Yang--Mills connection $D$ is also a critical point of the functional ${\rm YM}_{\rm BI}$ with respect to the metric $\tilde g$.
\end{proof}

\begin{Remark} The condition $n\ge 5$ is crucial in the previous proof because the Euler--Lagrange equations are conformally invariant in dimension $n=4$.
\end{Remark}

\section{The stress-energy tensor. Conservation law}\label{section3}

Motivated by Feynman's ideas on stationary electromagnetic field, in 1982 Baird and Eells introduced the stress-energy tensor associated to any smooth map $f\colon (M,g)\rightarrow (N,h)$ between two Riemannian manifolds, The stress-energy tensor is $S_f:=e(f)g-f^*h$, where $e(f)$ is the energy density of $f$. In the same spirit, to any $G$-connection $D$ one associates an analouguous $2$-tensor (related to the Yang--Mills--Born--Infeld functional) by (see also~\cite{Do}){\samepage
\begin{gather*}
S_D=\Big(\sqrt{1+\big\Vert R^D \big\Vert^2}-1\Big)g-\frac{1}{\sqrt{1+\big\Vert R^D \big\Vert^2}}R^D\odot R^D,
\end{gather*}
where $R^D\odot R^D$ is the symmetric product defined by $R^D\odot R^D=\big\langle i_XR^D,i_YR^D\big\rangle$.}

It is natural to look for the geometric interpretation of this tensor. There exists a variational interpretation which we shall explain in the following. Consider the following functional
\begin{gather*}
{\cal E}_D(g)=\int_M \Big(\sqrt{1+\big\Vert R^D \big\Vert^2}-1\Big)\vartheta_g.
\end{gather*}
The difference between this functional and ${\rm YM}_{\rm BI}$ is that ${\cal E}_D$ is defined on the space of smooth Riemannian metrics on the base manifold~$M$ and the connection~$D$ is fixed. In order to compute the rate of change of ${\cal E}_D(g)$ when the metric on the base manifold is changed, we consider a~smooth family of metrics~$g_s$ with $s\in (-\varepsilon, +\varepsilon)$, such that $g_0=g$. The ``tangent'' vector at $g$ to the curve of metrics~$g_s$ is denoted by $\delta g=\frac{d g_s}{d s}\big\vert_{s=0}$ and can be viewed as a smooth $2$-covariant symmetric tensor field on~$M$. Using the formulae obtained by Baird (see \cite{Ba})
\begin{gather*}
\frac{d\big\Vert R^D\big\Vert_{g_s}}{ds}\bigg\vert_{s=0}=-\big\langle R^D\odot R^D,\delta g\big\rangle,
\end{gather*}
and
\begin{gather*}
\frac{d}{ds}\vartheta_{g_s}\bigg\vert_{s=0}=\frac12\langle g,\delta g\rangle\vartheta_g
\end{gather*}
we obtain
\begin{gather*}
\frac{d{\cal E}_D(g_s)}{ds}\bigg\vert_{s=0}=\int_M\frac{1}{\sqrt{1+\big\Vert R^D \big\Vert^2}}\frac{d}{ds}\left(\frac12\big\Vert R^D \big\Vert^2\right)\bigg\vert_{s=0}\vartheta_g\\
\hphantom{\frac{d{\cal E}_D(g_s)}{ds}\bigg\vert_{s=0}=}{}
+\int_M \Big(\sqrt{1+\big\Vert R^D \big\Vert^2}-1\Big)\frac{d}{ds}\vartheta_{g_s}\bigg\vert_{s=0}\\
\hphantom{\frac{d{\cal E}_D(g_s)}{ds}\bigg\vert_{s=0}}{}
=\frac12\int_M\left\langle \Big(\sqrt{1+\big\Vert R^D \big\Vert^2}-1\Big)g-\frac{1}{\sqrt{1+\big\Vert R^D \big\Vert^2}}R^D\odot R^D,\delta g\right\rangle\vartheta_g\\
\hphantom{\frac{d{\cal E}_D(g_s)}{ds}\bigg\vert_{s=0}}{}
=\frac12\int_M\langle S_D,g\rangle\vartheta_g.
\end{gather*}

Recall now

\begin{Definition}A $G$-connection $D$ is said to satisfy a conservation law if $S_D$ is divergence free.
\end{Definition}

Concerning this notion we obtain the following result (see \cite{Do} for the general case of $F$-Yang--Mills fields).

\begin{Proposition}Any critical point of the functional ${\rm YM}_{\rm BI}$ is conservative.
\end{Proposition}

\begin{proof}The following formula for the divergence of the stress-energy tensor is true (see \cite{Do})
\begin{gather*}
\operatorname{div} S_D(X)=\left\langle \frac{1}{\sqrt{1+\big\Vert R^D \big\Vert^2}}\delta^D R^D-i_{\operatorname{grad}\big(\frac{1}{\sqrt{1+\Vert R^D \Vert^2}}\big)}R^D,i_XR^D\right\rangle\\
\hphantom{\operatorname{div} S_D(X)=}{} +\frac{1}{\sqrt{1+\big\Vert R^D \big\Vert^2}}\big\langle i_Xd^DR^D,R^D\big\rangle,
\end{gather*}
for any vector field $X$ on $M$. Using the Bianchi identity and the Euler--Lagrange equation of the functional ${\rm YM}_{\rm BI}$, we derive $\operatorname{div}S_D=0$.
\end{proof}

\section{The second variation formula}\label{section4}

In this section we obtain the second variation formula of the functional ${\rm YM}_{\rm BI}$. Let $(M,g)$ be an $n$-dimensional compact Riemannian manifold, $G$ a compact Lie group and $E$ a $G$-vector bundle over $M$. Let $D$ be a critical point of the functional ${\rm YM}_{\rm BI}$ and $D^t$ a smooth curve on ${\cal C}(E)$ such that $D^t=D+\alpha^t$, where $\alpha^t\in \Omega^1(g_E)$ for all $t\in (-\varepsilon,\varepsilon)$, and $\alpha^0=0$. The infinitesimal variation of the connection associated to $D^t$ at $t=0$ is
\begin{gather*}
B:=\frac{d\alpha^t}{dt}\bigg\vert_{t=0}\in \Omega (g_E).
\end{gather*}
According to \cite{BL}, we define the endomorphism ${\cal R}^D$ of $\Omega^1 (g_E)$ by
\begin{gather*}
{\cal R}^D(\varphi)(X):=\sum_{i=1}^n\big[R^D(e_i,X),\varphi (e_i)\big],
\end{gather*}
for $\varphi \in \Omega (g_E)$ and $X\in \Gamma(TM)$, where $\{e_i\}_{i=1}^n$ is a local orthonormal frame field on $(M,g)$. With these notations we have

\begin{Theorem}Let $(M,g)$ be an $n$-dimensional compact Riemannian manifold, $G$ a~compact Lie group and $E$ a $G$-vector bundle over~$M$. Let~$D$ be a critical point of ${\rm YM}_{\rm BI}$. The second variation of the functional ${\rm YM}_{\rm BI}$ is given by
\begin{gather*}
\frac {d^2}{dt^2}\bigg\vert_{t=0} {\rm YM}_{\rm BI}\big(D^t\big) = -\int_M\frac{1}{\big(1+\big\Vert R^D \big\Vert^2\big)^{3/2}}\big\langle d^DB,R^D\big\rangle^2 \vartheta_g\\
\hphantom{\frac {d^2}{dt^2}\bigg\vert_{t=0} {\rm YM}_{\rm BI}\big(D^t\big) =}{} +\int_M\frac{1}{\sqrt{1+\big\Vert R^D \big\Vert^2}} \big( \big\langle d^DB,d^DB\big\rangle+\big\langle B, {\cal R}^D(B)\big\rangle\big)\vartheta_g \\
\hphantom{\frac {d^2}{dt^2}\bigg\vert_{t=0} {\rm YM}_{\rm BI}\big(D^t\big)}{}
=\int_M \big\langle B,{\cal S}^D(B)\big\rangle \vartheta_g,
\end{gather*}
where ${\cal S}^D$ is a differential operator acting on $\Omega (g_E)$ defined by
\begin{gather*}
{\cal S}^D(B)=-\delta^D\left( \frac{1}{\big(1+\big\Vert R^D \big\Vert^2\big)^{3/2}}\big\langle d^DB,R^D\big\rangle^2 \right)\\
\hphantom{{\cal S}^D(B)=}{} + \delta^D\left( \frac{1}{\sqrt{1+\big\Vert R^D \big\Vert^2}}d^DB\right)++\frac{1}{\sqrt{1+\big\Vert R^D \big\Vert^2}}{\cal R}^D(B).
\end{gather*}
\end{Theorem}

\begin{proof}As $R^{D^t}=R^D + d^D\alpha^t + \frac12 \big[\alpha^t \wedge \alpha^t\big]$ and $\alpha^0=0$ we obtain that
\begin{gather*}
\frac {d^2}{dt^2}\bigg\vert_{t=0} \left(\frac 12\big\|R^{D^t}\big\|^2\right) =\big\langle d^DC+[B,B], R^D\big\rangle+\big\langle d^DB,d^DB\big\rangle,
\end{gather*}
where $C:=\frac{d^2}{dt^2}\big\vert_{t=0}\alpha^t$. Thus we obtain
\begin{gather*}
\frac {d^2}{dt^2}\bigg\vert_{t=0}{\rm YM}_{\rm BI}\big(D^t\big) =\frac {d}{dt}\bigg\vert_{t=0} \int_M\frac12\frac{1}{\sqrt{1+\big\Vert R^D \big\Vert^2}}\frac {d}{dt}\big\|R^{D^t}\big\|^2 \vartheta_g\\
\hphantom{\frac {d^2}{dt^2}\bigg\vert_{t=0}{\rm YM}_{\rm BI}\big(D^t\big)}{}
 = -\frac14 \int_M\frac{1}{\big(1+\big\Vert R^D \big\Vert^2\big)^{3/2}}\left(\frac {d}{dt}\bigg \vert_{t=0}\|R^{D^t}\|^2\right)^2\vartheta_g\\
\hphantom{\frac {d^2}{dt^2}\bigg\vert_{t=0}{\rm YM}_{\rm BI}\big(D^t\big)=}{}
 +\frac12 \int_M\frac{1}{\sqrt{1+\big\Vert R^D \big\Vert^2}}\frac {d^2}{dt^2}\bigg\vert_{t=0}\big\|R^{D^t}\big\|^2 \vartheta_g \\
\hphantom{\frac {d^2}{dt^2}\bigg\vert_{t=0}{\rm YM}_{\rm BI}\big(D^t\big)}{}
=-\int_M\frac{1}{\big(1+\big\Vert R^D \big\Vert^2\big)^{3/2}}\big\langle d^DB,R^D\big\rangle ^2 \vartheta_g\\
\hphantom{\frac {d^2}{dt^2}\bigg\vert_{t=0}{\rm YM}_{\rm BI}\big(D^t\big)=}{}
+\int_M\frac{1}{\sqrt{1+\big\Vert R^D \big\Vert^2}}\big(\big\langle d^DC+[B,B], R^D\big\rangle+\big\langle d^DB,d^DB\big\rangle \big) \vartheta_g.
 \end{gather*}
On the other hand, since $D$ is a critical point of the functional ${\rm YM}_{\rm BI}$, we have
 \begin{gather*}
\int_M\frac{1}{\sqrt{1+\big\Vert R^D \big\Vert^2}}\big\langle d^DC, R^D\big\rangle \vartheta_g=\int_M\left<C,\delta^D\left(\frac{1}{\sqrt{1+\big\Vert R^D \big\Vert^2}}R^D\right)\right>\vartheta_g=0.
 \end{gather*}
 Finally, one can prove that
 \begin{gather*}
\big\langle [B\wedge B],R^D\big\rangle=\big\langle B,{\cal R}^D(B)\big\rangle.
\end{gather*}
Indeed
\begin{gather*}
\big\langle [B\wedge B],R^D\big\rangle =\sum_{i<j}\big\langle [B\wedge B](e_i,e_j),R^D(e_i,e_j)\big\rangle\\
\hphantom{\big\langle [B\wedge B],R^D\big\rangle}{} =\sum_{i<j}\big\langle [B(e_i),B(e_j)]-[B(e_j),B(e_i)],R^D(e_i,e_j)\big\rangle\\
\hphantom{\big\langle [B\wedge B],R^D\big\rangle}{} =2\sum_{i<j}\big\langle [B(e_i),B(e_j)],R^D(e_i,e_j)\big\rangle =\sum_{i,j=1}^n\big\langle B(e_i),[B(e_j),R^D(e_i,e_j)]\big\rangle\\
\hphantom{\big\langle [B\wedge B],R^D\big\rangle}{} =\sum_{i=1}^n\big\langle B(e_i),{\cal R}^D(e_i)\big\rangle=\big\langle B,{\cal R}^D(B)\big\rangle,
\end{gather*}
and thus we obtain the second variation formula.
\end{proof}

The index, nullity and stability of a critical point of ${\rm YM}_{\rm BI}$ can be defined in the same way as in the case of Yang--Mills connection (see~\cite{BL}) but is rather difficult to analyse them because the form of ${\cal S}^D$ is much more complicated compared with the case of Yang--Mills connections.

\subsection*{Acknowledgements}

The author thank the referees for very carefully reading a first version of the paper and for their useful suggestions. This work is partially supported by a Grant of Ministry of Research and Innovation, CNCS - UEFISCDI, Project Number PN-III-P4-ID-PCE-2016-0065, within PNCDI~III.

\pdfbookmark[1]{References}{ref}
\LastPageEnding


\begin{thebibliography}{99}
\footnotesize\itemsep=0pt

\bibitem{Ba}
Baird P., Stress-energy tensors and the {L}ichnerowicz {L}aplacian,
 \href{https://doi.org/10.1016/j.geomphys.2008.05.008}{\textit{J.~Geom. Phys.}} \textbf{58} (2008), 1329--1342.

\bibitem{BL}
Bourguignon J.P., Lawson Jr. H.B., Stability and isolation phenomena for
 {Y}ang--{M}ills fields, \href{https://doi.org/10.1007/BF01942061}{\textit{Comm. Math. Phys.}} \textbf{79} (1981),
 189--230.

\bibitem{Do}
Dong Y., Wei S.W., On vanishing theorems for vector bundle valued {$p$}-forms
 and their applications, \href{https://doi.org/10.1007/s00220-011-1227-8}{\textit{Comm. Math. Phys.}} \textbf{304} (2011),
 329--368, \href{https://arxiv.org/abs/1003.3777}{arXiv:1003.3777}.

\bibitem{Uh}
Uhlenbeck K.K., Connections with {$L^{p}$} bounds on curvature, \href{https://doi.org/10.1007/BF01947069}{\textit{Comm.
 Math. Phys.}} \textbf{83} (1982), 31--42.

\end{thebibliography}
\end{document}